\documentclass{amsart}
\usepackage{yufei}
\usepackage[utf8]{inputenc}
\usepackage{cite}
\title{On the first two eigenvalues of regular graphs}
\author{Shengtong Zhang}
\address{Department of Mathematics, Stanford University, Stanford, CA 94305, USA}
\email{stzh1555@stanford.edu}


\begin{document}
\maketitle
\begin{abstract}
    Let $G$ be a regular graph with $m$ edges, and let $\mu_1, \mu_2$ denote the two largest eigenvalues of $A_G$, the adjacency matrix of $G$. We show that, if $G$ is not complete, then
    $$\mu_1^2 + \mu_2^2 \leq \frac{2(\omega - 1)}{\omega} m$$
    where $\omega$ is the clique number of $G$. This confirms a conjecture of Bollob\'{a}s and Nikiforov for regular graphs. We also show that equality holds if and only if $G$ is either a balanced Tur\'{a}n graph or the disjoint union of two balanced Tur\'{a}n graphs of the same size.
\end{abstract}
\section{Introduction}
Throughout the paper, let $G = (V, E)$ be a simple graph with $n(G)$ vertices and $m(G)$ edges. Let $A_G$ denote its adjacency matrix, and let $\mu_1(G), \mu_2(G), \cdots$ denote the largest, second largest, etc. eigenvalue of $A_G$. The well-known Tur\'{a}n's theorem can be concisely stated as  
$$2m(G) \leq \frac{\omega(G) - 1}{\omega(G)} n(G)^2$$
where $\omega(G)$ is the clique number, the size of the largest clique in $G$. 

In 2002, Nikiforov \cite{N02}, strengthening previous work of Wilf \cite{W86}, proved a spectral analog of Tur\'{a}n's theorem. Nikiforov showed that
\begin{equation}
\label{eq:spectral-turan}
\mu_1(G)^2 \leq \frac{2(\omega(G) - 1)}{\omega(G)} m(G).    
\end{equation}
Combined with the familiar bound $\mu_1(G) \geq \frac{2m(G)}{n(G)}$, this result implies Tur\'{a}n's theorem. Many extensions of this inequality are known, such as \cite{N06, BN07}. In fact,  Nikiforov's result spurred a large body of work that studies spectral analogs of the Tur\'{a}n number $\mathsf{ex}(n, F)$. We refer to the fantastic survey by Liu and Ning \cite{LN23} for more details.

In 2007,  Bollob\'{a}s and Nikiforov \cite{BN07} conjectured that something stronger should hold. 
\begin{conjecture}
\label{conj:NB}
Suppose $G$ is a graph with $m(G)$ edges, clique number $\omega$, and is not the complete graph $K_{\omega}$. Then 
$$\mu_1(G)^2 + \mu_2(G)^2 \leq \frac{2(\omega - 1)}{\omega} m(G).$$
\end{conjecture}
The complete graph is excluded since it satisfies $\mu_1(K_{\omega}) = \omega - 1$ and $\mu_2(K_{\omega}) = -1$. On the other hand, if $G$ is not complete, then the empty graph $\bar{P_2}$ on two vertices is an induced subgraph of $G$, so Cauchy interlacing implies that $\mu_2(G) \geq \mu_2(\bar{P_2}) = 0$. So we can rewrite this conjecture as
$$\mu_1(G)^2 + \mu_2(G)^2\mathbf{1}_{\mu_2(G) \geq 0} \leq \frac{2(\omega - 1)}{\omega} m(G)$$
for all graphs $G$.

The goal of this paper is to prove this conjecture for regular graphs.
\begin{theorem}
    \label{thm:BN-regular}
    Suppose $G$ is a regular graph with clique number $\omega$, and is not the complete graph $K_{\omega}$. Then 
    $$\mu_1(G)^2 + \mu_2(G)^2 \leq \frac{2(\omega - 1)}{\omega} m(G).$$
    The equality holds if and only if 1) $G$ is a Tur\'{a}n graph $T(n, \omega)$ with $n$ divisible by $\omega$ or 2) $G$ is the disjoint union $T(n / 2, \omega) \sqcup T(n / 2, \omega)$ of two Tur\'{a}n graphs  with $n$ divisible by $2\omega$.
\end{theorem}
Before we dive into the proof, we summarize past works on this conjecture. A significant breakthrough occured in 2015, as Ando and Lin \cite{AL15} proved this conjecture with $\omega(G)$ replaced by $\chi(G)$, the chromatic number of $G$. In fact, they showedthe stronger result
$$\sum_{i: \mu_i(G) \geq 0} \mu_i(G)^2 \leq \frac{2(\chi(G) - 1)}{\chi(G)} m(G).$$
 for any graph $G$. As a corollary, the conjecture of Bollob\'{a}s and Nikiforov holds for all weakly perfect graphs $G$. However, for general graphs $G$, it is a classical result of Erd\H{o}s that $\chi(G)$ is not bounded by any function of $\omega(G)$.

Recently, Guo and Spiro \cite{GS22} replaced $\chi(G)$ with the fractional chromatic number $\chi_f(G)$. More recently, Coutinho and Spier \cite{CS23} obtained a further improvement, replacing $\chi_f(G)$ with the vector chromatic number $\chi_{vec}(G)$ of $G$. This is still insufficient to prove \cref{conj:NB}. For example, when $G = C_5$, we have $\chi_{vec}(G) = \sqrt{5}$ while $\omega(G) = 2$.

Another breakthrough occured in 2021, as Lin, Ning and Wu \cite{LNW21} proved the conjecture of Bollob\'{a}s and Nikiforov when $\omega = 2$. Their method exploits the fact that the number of $K_3$ in $G$ is equal to
$$\frac{1}{6} \sum_{i = 1}^{n(G)} \mu_i(G)^3.$$
However, no such identity is available for the number of $K_4$. As far as I know, the conjecture of Bollob\'{a}s and Nikiforov is still open when $\omega = 3$.

Let's now look at the proof to \cref{thm:BN-regular}. For convenience, we make all dependence on $G$ implicit, e.g. abbreviate all $m(G)$ with $m$. We also fix a few linear algebraic notations.

\begin{itemize}
    \item Let $\bv_i$ be any normalized real eigenvector corresponding to $\mu_i$ in $A_G$.
    \item Let $\bone$ denote the all one vector, and let $J$ denote the all one matrix.
    \item For a vector $\bv \in \CC^V$ and a vertex $i \in V$, let $\bv(i)$ denote the value of $\bv$ at index $i$.
    \item For vectors $\bx$ and $\by$ in $\CC^V$, we denote by $\bx \circ \by$ their elementwise product, with $\bx \circ \by(i) = \bx(i) \by(i)$ for each $i \in V$.
\end{itemize}
\section{Recap of Previous Work}
Our proof combines the technique of Nikiforov \cite{N02} and Ando-Lin \cite{AL15}.

Nikiforov's main tool is the celebrated Motzkin-Straus inequality \cite{MS65}. We consider the matrix 
$$K_G = \frac{\omega - 1}{\omega}J - A_G.$$ 
The Motzkin-Straus inequality states that
\begin{theorem}
    \label{thm:Niki}
    Let $\bx \in \RR_{\geq 0}^{V}$ be any vector with non-negative entries. Then
    $$\bx^T A_G \bx \leq \frac{\omega - 1}{\omega} (\mathbf{1}^T \bx)^2.$$
    In other words
    $$\bx^T K_G \bx \geq 0.$$  
\end{theorem}
We briefly recap how \eqref{eq:spectral-turan} is derived from this theorem. We take $\bx = \bv_1 \circ \bv_1$, and apply Cauchy-Schwarz on the left hand side to get
$$\bx^T A_G \bx \geq \frac{1}{2m} \left(\bv_1^T A_G \bv_1\right)^2 = \frac{\mu_1^2}{2m}.$$
The spectral Tur\'{a}n theorem now follows by rearranging.

Ando-Lin's idea is a matrix decomposition trick. Here, we present a modification of their argument specifically tailored for \cref{conj:NB}. 
\begin{lemma}
\label{lem:AL}
Let $G$ be any graph. Suppose $G$ satisfies the following inequality for some $r > 1$
\begin{equation}
\label{eq:AL-1}
\sum_{ij \in E} (\mu_1 \bv_1(i) \bv_1(j) + \mu_2 \bv_2(i) \bv_2(j))^2 \leq \frac{r - 1}{r} \sum_{i, j \in V} (\mu_1 \bv_1(i) \bv_1(j) + \mu_2 \bv_2(i) \bv_2(j))^2    
\end{equation}
or equivalently, for $K_G = \frac{r - 1}{r} J -A_G$, we have
\begin{equation}
\label{eq:AL-2}
\mu_1^2 (\bv_1 \circ \bv_1)^T K_G (\bv_1 \circ \bv_1) + \mu_2^2 (\bv_2 \circ \bv_2)^T K_G (\bv_2 \circ \bv_2) + 2 \mu_1\mu_2 (\bv_1 \circ \bv_2)^T K_G (\bv_1 \circ \bv_2) \geq 0.
\end{equation}
Then $G$ satisfies 
$$\mu_1^2 + \mu_2^2 \leq \frac{r - 1}{r} \cdot 2m.$$
\end{lemma}
\begin{proof}
First, let's justify why \eqref{eq:AL-1} and \eqref{eq:AL-2} are equivalent. Expanding the left hand side of \eqref{eq:AL-1}, we have
\begin{align*}
    &\sum_{ij \in E} (\mu_1 \bv_1(i) \bv_1(j) + \mu_2 \bv_2(i) \bv_2(j))^2 \\
    =& \sum_{ij \in E} \mu_1^2 \bv_1(i)^2 \bv_1(j)^2 + 2\mu_1\mu_2 \bv_1(i) \bv_1(j) \bv_2(i) \bv_2(j) + \mu_2^2 \bv_2(i)^2 \bv_2(j)^2 \\
    =& \mu_1^2 \sum_{ij \in E} (\bv_1 \circ \bv_1)(i) (\bv_1 \circ \bv_1)(j) + 2\mu_1\mu_2 \sum_{ij \in E} (\bv_1 \circ \bv_2)(i) (\bv_1 \circ \bv_2)(j) + \mu_2^2 \sum_{ij \in E} (\bv_2 \circ \bv_2)(i) (\bv_2 \circ \bv_2)(j) \\
    =& \mu_1^2 (\bv_1 \circ \bv_1)^T A_G (\bv_1 \circ \bv_1) + 2\mu_1\mu_2 (\bv_1 \circ \bv_2)^T A_G (\bv_1 \circ \bv_2) + 
    \mu_2^2 (\bv_2 \circ \bv_2)^T A_G (\bv_2 \circ \bv_2).
\end{align*}
Similarly, the right hand side of \eqref{eq:AL-1} is equal to
$$ \frac{r - 1}{r} \left(\mu_1^2 (\bv_1 \circ \bv_1)^T J (\bv_1 \circ \bv_1) + 2\mu_1\mu_2 (\bv_1 \circ \bv_2)^T J (\bv_1 \circ \bv_2) + 
\mu_2^2 (\bv_2 \circ \bv_2)^T J (\bv_2 \circ \bv_2)\right).$$
Thus, the difference between the RHS and the LHS of the \eqref{eq:AL-1} is precisely the LHS of \eqref{eq:AL-2}.

Now, we use Ando-Lin's method to prove that \eqref{eq:AL-1} implies the result of the lemma. It is a familiar fact that we can write $A_G$ has
$$A_G = \sum_{i = 1}^{n} \mu_i \bv_i \bv_i^T.$$
Following Ando and Lin, we write
$$X = \mu_1 \bv_1 \bv_1^T + \mu_2 \bv_2 \bv_2^T, Y = A_G - X = \sum_{i \geq 3} \mu_i \bv_i \bv_i^T.$$
Then \cref{eq:AL-1} becomes
$$\sum_{ij \in E} X_{ij}^2 \leq \frac{r - 1}{r}\sum_{i, j \in V} X_{ij}^2$$
or
$$\sum_{ij \in E} X_{ij}^2 \leq (r - 1)\sum_{ij \notin E} X_{ij}^2$$
As $X + Y = A_G$, we have $Y_{ij} = -X_{ij}$ whenever $ij \notin E$. So
$$\sum_{ij \notin E} Y^2_{ij} = \sum_{ij \notin E} X^2_{ij} \geq \frac{1}{r - 1} \sum_{ij \in E} X^2_{ij}.$$
Furthermore, the inner product $\langle X, Y \rangle$ is equal to $0$, which implies
$$\sum_{ij \in E} X_{ij} Y_{ij} = - \sum_{ij \notin E} X_{ij} Y_{ij} = \sum_{ij \notin E} X_{ij}^2.$$
Via Cauchy-Schwarz inequality, we get
$$\sum_{ij \in E} Y^2_{ij} \geq \frac{(\sum_{ij \notin E} X_{ij}^2)^2}{\sum_{ij \in E} X_{ij}^2} \geq \frac{1}{r - 1} \sum_{ij \notin E} X_{ij}^2.$$
Adding the two inequalities, we conclude that
$$\sum_{i, j \in V} Y^2_{ij} \geq \frac{1}{r - 1} \sum_{i,j \in V} X_{ij}^2.$$
Finally, note that $\sum_{ij} X_{ij}^2 = \langle X, X \rangle = \mu_1^2 + \mu_2^2$, and similarly $\sum_{ij} Y_{ij}^2 = \sum_{i \geq 3} \mu_i^2$. Thus we get
$$\mu_1^2 + \mu_2^2 \leq (r - 1)\sum_{i \geq 3} \mu_i^2.$$
As $\mu_1^2 + \mu_2^2 + \sum_{i \geq 3} \mu_i^2 = 2m$, we conclude that
$$\mu_1^2 + \mu_2^2 \leq \frac{2(r - 1)}{r} m$$
as desired.
\end{proof}
\section{Proof of Conjecture 1.1 for regular graphs}
Suppose that $G$ has clique number $\omega$, and is $d$-regular on $n \geq \omega + 1$ vertices. Then $\mu_1 = d$ and $\mu_2 \geq 0$. In this case, we can take $\bv_1 = \frac{1}{\sqrt{n}} \bone$.

Suppose for the sake of contradiction that $G$ does not satisfy \cref{conj:NB}. We apply \cref{thm:Niki} on two vectors
$$\bx_+ = (\sqrt{\mu_1} \bv_1 + \sqrt{\mu_2} \bv_2) \circ (\sqrt{\mu_1} \bv_1 + \sqrt{\mu_2} \bv_2), \bx_- = (\sqrt{\mu_1} \bv_1 - \sqrt{\mu_2} \bv_2) \circ (\sqrt{\mu_1} \bv_1 - \sqrt{\mu_2} \bv_2).$$
Clearly these two vectors are non-negative, so we get
$$\bx_+^T K_G \bx_+ + \bx_-^T K_G \bx_- \geq 0.$$
We can expand this expression by observing that
$$\bx_{\pm} = \mu_1 \bv_1 \circ \bv_1 + \mu_2 \bv_2 \circ \bv_2 \pm 2 \sqrt{\mu_1\mu_2} \bv_1 \circ \bv_2.$$
So we get
\begin{align*}
    0 \leq& \frac{1}{2}(\bx_+^T K_G \bx_+ + \bx_-^T K_G \bx_-) \\
    =& \mu_1^2 (\bv_1 \circ \bv_1)^T K_G (\bv_1 \circ \bv_1) + \mu_2^2 (\bv_2 \circ \bv_2)^T K_G (\bv_2 \circ \bv_2) + 2 \mu_1\mu_2 (\bv_1 \circ \bv_1)^T K_G (\bv_2 \circ \bv_2) \\
    &+ 4 \mu_1\mu_2 (\bv_1 \circ \bv_2)^T K_G (\bv_1 \circ \bv_2).
\end{align*}
For the sake of contradiction, we assume that $G$ does not satisfy \cref{conj:NB}. Then \cref{lem:AL} with $r = \omega$ tells us that
$$\mu_1^2 (\bv_1 \circ \bv_1)^T K_G (\bv_1 \circ \bv_1) + \mu_2^2 (\bv_2 \circ \bv_2)^T K_G (\bv_2 \circ \bv_2) + 2 \mu_1\mu_2 (\bv_1 \circ \bv_2)^T K_G (\bv_1 \circ \bv_2) < 0.$$
Comparing the two inequalities, we must have
$$2 \mu_1 \mu_2 (\bv_1 \circ \bv_2)^T K_G (\bv_1 \circ \bv_2) + 2\mu_1\mu_2 (\bv_1 \circ \bv_1)^T K_G (\bv_2 \circ \bv_2) > 0$$
or
$$(\bv_1 \circ \bv_2)^T K_G (\bv_1 \circ \bv_2) + (\bv_1 \circ \bv_1)^T K_G (\bv_2 \circ \bv_2) > 0.$$
The discussion above applies to a general graph $G$. We now utilize the specific condition that $\bv_1 = \frac{1}{\sqrt{n}} \bone$ and $K_G = \frac{\omega - 1}{\omega} J - A_G$. Also note that $\langle \bv_2,\bone \rangle = 0$, so $\bv_2^T J \bv_2 = 0$. We compute that
$$(\bv_1 \circ \bv_2)^T K_G (\bv_1 \circ \bv_2) = \frac{1}{n} \bv_2^T K_G \bv_2 = -\frac{1}{n} \mu_2$$
and
\begin{align*}
(\bv_1 \circ \bv_1)^T K_G (\bv_2 \circ \bv_2) &= \frac{1}{n} \bone^T K_G (\bv_2 \circ \bv_2)  \\
&= \frac{1}{n} \left(\frac{\omega - 1}{\omega} n - d\right) \bone^T (\bv_2 \circ \bv_2) \\
&= \frac{1}{n} \left(\frac{\omega - 1}{\omega} n - d\right) \norm{\bv_2}^2 = \frac{1}{n} \left(\frac{\omega - 1}{\omega} n - d\right).
\end{align*}
Thus we conclude that
$$\mu_2 < \frac{\omega - 1}{\omega} n - d.$$
Combining with the fact that $\mu_2 \leq \mu_1 = d$, we conclude that
$$\mu_2^2 < \frac{\omega - 1}{\omega} n d - d^2 = \frac{\omega - 1}{\omega} \cdot 2m - \mu_1^2$$
so $G$ must satisfy \cref{conj:NB}, contradiction.

\section{The case of equality}
In this section, we determine all regular graphs $G \neq K_{\omega}$ for which the equality
$$\mu_1^2 + \mu_2^2 = \frac{2(\omega - 1)}{\omega}m$$
holds. 

For such graphs $G$, we must have
$$\mu_1^2 (\bv_1 \circ \bv_1)^T K_G (\bv_1 \circ \bv_1) + \mu_2^2 (\bv_2 \circ \bv_2)^T K_G (\bv_2 \circ \bv_2) + 2 \mu_1\mu_2 (\bv_1 \circ \bv_2)^T K_G (\bv_1 \circ \bv_2) \leq 0.$$
Indeed, if the left-hand side is greater than $0$, then the assumption in \cref{lem:AL} holds for some $r < \omega$, so we have $\mu_1^2 + \mu_2^2 < \frac{2(\omega - 1)}{\omega}m$.

Following the analysis in the previous section, we must have
$$0 \leq \mu_2 \leq \frac{\omega - 1}{\omega} n - d.$$
Furthermore, we have $\mu_2 \leq d$. On the other hand, we have $\mu_2^2 = (\frac{\omega - 1}{\omega} n - d)d$. So one of the two cases below must hold.

\textbf{Case 1}: $\frac{\omega - 1}{\omega} n - d = 0$. In this case, $G$ is a $K_{\omega + 1}$-free graph on $n$ vertices with $\frac{\omega -  1}{2\omega} n^2$ edges, so $G$ must be a Tur\'{a}n graph $T(n, \omega)$ with $n$ divisible by $\omega$.

\textbf{Case 2}: $\mu_2 = \frac{\omega - 1}{\omega} n - d = d$. As $\mu_2 = \mu_1$, $G$ must have two connected components $G_1, G_2$, both regular with degree $d = \frac{\omega - 1}{2\omega} n$. By Tur\'{a}n's theorem applied to each $G_i$, this is possible if and only if $G_1$ and $G_2$ are Tur\'{a}n graph $T(n / 2,\omega)$ with $n$ divisible by $2\omega$.

\section{Difficulties Ahead}
To finish our discussion, we consider some difficulties in applying the methods in this paper to the non-regular case of \cref{conj:NB}.

A prerequisite for an inequality-based approach to succeed is that, for any $G$ such that \cref{conj:NB} is tight, any inequality involved should likely be tight. Unfortunately, the key inequality in our argument
\begin{align*}
0 \leq& \mu_1^2 (\bv_1 \circ \bv_1)^T K_G (\bv_1 \circ \bv_1) + \mu_2^2 (\bv_2 \circ \bv_2)^T K_G (\bv_2 \circ \bv_2) + 2 \mu_1\mu_2 (\bv_1 \circ \bv_1)^T K_G (\bv_2 \circ \bv_2) \\
    &+ 4 \mu_1\mu_2 (\bv_1 \circ \bv_2)^T K_G (\bv_1 \circ \bv_2)    
\end{align*}
is not tight for many tight examples of \cref{conj:NB}. One such examples is $G = T(n_1, \omega) \sqcup T(n_2, \omega)$, where $n_1 > n_2$ are both divisible by $\omega$. Another example is $G = P_5$ as discovered by Lin, Ning and Wu \cite{LNW21}. So additional ideas might be needed to completely resolve Bollob\'{a}s and Nikiforov's conjecture.

\section*{Acknowledgements}
The author thanks the anonymous referee, Dr. Jonathan Tidor, Prof. Hao Huang, Prof. Bo Ning, Dr. Clive Elphick, Prof. Huiqiu Lin and Prof. Vladimir Nikiforov for many valuable suggestions on this paper. The author is supported by the Craig Franklin Fellowship in Mathematics at Stanford University.

\bibliographystyle{plain}

\end{document}